\documentclass[12pt]{amsproc}
 \usepackage[margin=1in]{geometry}
\usepackage{setspace,fullpage}
\usepackage{graphicx}
\usepackage[nice]{nicefrac}
\usepackage{amssymb}
\usepackage{amsmath,amsthm,amscd,amssymb, mathrsfs}

\usepackage{latexsym}

\numberwithin{equation}{section}

\theoremstyle{plain}
\newtheorem{theorem}{Theorem}[section]
\newtheorem{lemma}[theorem]{Lemma}
\newtheorem{corollary}[theorem]{Corollary}

\theoremstyle{definition}
\newtheorem{definition}[theorem]{Definition}

\theoremstyle{remark}

\newtheorem{case[theorem]}{Case}

\title[\parbox{14cm}{\centering{Restriction operators acting on radial functions on vector spaces over finite fields \hspace{1in}}} \quad]{Restriction operators acting on radial functions on vector spaces over finite fields }
\author{ Doowon Koh }

\address{Department of Mathematics\\
Chungbuk National University \\
Cheongju city, Chungbuk-Do 361-763 Korea}
\email{koh131@chungbuk.ac.kr}

\thanks{
This research was supported by Basic Science Research Program through the National Research Foundation of Korea(NRF) funded by the Ministry of Education, Science and Technology(2012010487)
}


\begin{document}

\begin{abstract}  We stduy $L^p-L^r$ restriction estimates for algebraic varieties $V$ in the case when restriction operators act on radial functions in the finite field setting.  
We show that  if  the varieties $V$ lie in odd dimensional vector spaces over finite fields, then the conjectured restriction estimates are possible for all radial test functions.
In addition,  it is proved that if the varieties $V$ in even dimensions have few intersection points with the sphere of zero radius, the same conclusion as in odd dimensional case can be also obtained.
\end{abstract} 
\maketitle

\section{Introduction}

Let $V$ be a subset of ${\mathbb R}^d, d\geq 2,$ and $d\sigma$ a positive measure
supported on $V$. Then, one may ask that for which values of $p$ and $r$ does the following inequality 
$$ \|\widehat{f}\|_{L^r(V, d\sigma)}
\le C_{p,r,d} ~\|f\|_{L^p(\mathbb R^d)}\quad \text{for all} \quad 
 f \in L^p(\Bbb R^d)$$
hold? This problem is known as the restriction problem in Euclidean space  and  it was first posed by E.M. Stein in 1967.
The restriction problem for the circle and the parabola in the plane was completely solved by Zygmund (\cite{Zy74})
and the problem for cones in  three and four  dimesions was also established by Barcelo (\cite{Ba85}) and Wolff (\cite{Wo01}) respectively.
However, this problem is still open in other higher dimensions and it have been considered as one of the most important, difficult problems in harmonic analysis.
We refer  the reader to \cite{Ta04},\cite{Wo03}, \cite{St93},\cite{Bo91}, \cite{Fe70}, \cite{Ta03}  for  further discussion  and recent progress  on the restriction problem.
\vskip 0.1in
As an analog of the Euclidean restriction problem, Tao and Mockenhaupt (\cite{MT04}) recently reformulated and studied the restriction problem for various algebraic varieties in the finite field setting.
In this introduction we review the definition, a conjecture, and known results on the restriction problem for algebraic varieties in $d$-dimensional vector spaces over finite fields.
Let $\mathbb F_q^d$ be a $d$-dimensional vector space over the finite field $\mathbb F_q$ with $q$ elements.
We endow this space with a counting measure $dm$. Thus, if $f: \mathbb F_q^d \to \mathbb C$, then its integral over $\mathbb F_q^d$ is given by
$$ \int_{\mathbb F_q^d} f(m) dm= \sum_{m\in \mathbb F_q^d} f(m).$$
We denote by $\mathbb F_{q*}^d$ the dual space of $\mathbb F_q^d.$ We endow the dual space $\mathbb F_{q*}^d$ with a normalized counting measure $dx$. Hence, given a function $g:\mathbb F_{q*}^d \to \mathbb C$, we define its integral
$$\int_{\mathbb F_{q*}^d} g(x) dx =\frac{1}{q^d} \sum_{x\in \mathbb F_{q*}^d} g(x).$$
Recall that the space $\mathbb F_q^d$ is isomorphic to its dual space $\mathbb F_{q*}^d$ as an abstract group. Also recall that if $f$ is a complex-valued function on $\mathbb F_{q}^d$, its Fourier transform, denoted by $\widehat{f}$, is actually defined on its dual space $\mathbb F_{q^*}^d$:
$$ \widehat{f}(x)=\int_{\mathbb F_q^d} \chi(-m\cdot x) f(m) dm= \sum_{m\in \mathbb F_q^d} \chi(-m\cdot x) f(m),$$
where $\chi$ denotes a nontrivial additive character of $\mathbb F_q.$ Let $V$ be an algebraic variety in the dual space $\mathbb F_{q*}^d.$  Throughtout the paper we always assume that $|V|\sim q^{d-1}.$
Namely, we view the variety $V$ as a hypersurface in $\mathbb F_{q*}^d.$
Recall that a normalized surface measure on $V$, denoted by $d\sigma$, is defined by the relation
$$ \int g(x) d\sigma(x) = \frac{1}{|V|}\sum_{x\in V} g(x),$$
where $g:\mathbb F_{q*}^d \to \mathbb C.$
With notation above, the restriction problem for the variety $V$ is to determine $1\leq p, r\leq \infty$ such that
the following restriction estimate holds:
\begin{equation}\label{restriction} \|\widehat{f}\|_{L^r(V, d\sigma)} \leq C \|f\|_{L^p(\mathbb F_q^d, dm)} \quad\mbox{for all functions}~~f:\mathbb F_q^d \to \mathbb C,\end{equation}
where the constant $C>0$ is independent of functions $f$ and the size of the underlying finite field $\mathbb F_q.$ We shall use the notation $R(p\to r)\lesssim 1$ to indicate that the restriction estimate (\ref{restriction}) holds.
 By duality, the inequality (\ref{restriction}) is same as the following extension estimate:
$$ \|(gd\sigma)^\vee\|_{L^{p^\prime}(\mathbb F_q^d, dm)} \leq C \|g\|_{L^{r^\prime}(V, d\sigma)}.$$ 
Mockenhaupt and Tao (\cite{MT04}) observed that the necessary conditions for the inequality (\ref{restriction}) take
\begin{equation}\label{conjecture} 1\leq p, r\leq \infty, \quad \frac{1}{p}\geq \frac{d+1}{2d} \quad \mbox{and}\quad \frac{d}{p}+\frac{d-1}{r}\geq d.\end{equation}
Namely, $R(p\to r)\lesssim 1$ only if $(1/p, 1/r)$ lies in the convex hull of points
\begin{equation}\label{necessary}
(1,0), (1,1), \left(\frac{d+1}{2d}, 1\right), ~~\mbox{and}~~ \left( \frac{d+1}{2d}, \frac{1}{2}\right).
\end{equation}
They also proved that the necessary conditions (\ref{necessary}) are in fact sufficient  for $A(p\to r)\lesssim 1$ if $V$ is the parabola in $\mathbb F_{q*}^2.$ In \cite {KS12}, Koh and Shen generalized the result to general algebraic cuves in two dimensions. However, in higher dimensions the restriction problem has not been solved and the known results are even weaker than those in Euclidean space.
The currently best known results on restriction problems for paraboloids in $\mathbb F_{q*}^d$ are due to A. Lewko and M.Lewko (\cite{LL10}). They established certain endpoint restriction estimates for paraboloids, which slightly improve on the previously known results  by Mockenhaupt and Tao(\cite{MT04}) in three dimensions and those by Iosevich and Koh (\cite{IK09}) in higher dimensions. More precisely, the following theorem was proved by them.
\begin{theorem}\label{Lewko} Let $V=\{x\in \mathbb F_{q^*}^d: x_1^2+\cdots+ x_{d-1}^2=x_d\}$ be the paraboloid.  
If $d=3$ and $-1\in \mathbb F_{q*}$ is not a square, then $A(p\to r)\lesssim 1$ whenever $(1/p, 1/r)$ lies in the convex hull of points 
$$(1,0), (1,1), (13/18 , 1), (13/18, 1/2), ~~\mbox{and}~~ (3/4, 3/8).$$
Moreover, if $d\geq 4$ is even or if $d=4k+3$ for some $k\in\mathbb N$ and $-1\in\mathbb F_{q*}$ is not a square, then  $A(p\to r)\lesssim 1$ whenever $(1/p, 1/r)$ is contained in the convex hull of points
\begin{equation}\label{best} (1,0), (1,1), \left(\frac{d^2+2d-2}{2d^2}, 1\right), \left(\frac{d^2+2d-2}{2d^2}, \frac{1}{2}\right),~~\mbox{and} ~~\left(\frac{3}{4}, \frac{d+2}{4d}\right).
\end{equation}
\end{theorem}

As we shall see, our main results below imply that the restriction conjecture (\ref{conjecture}) for paraboloids in $\mathbb F_{q*}^d$ holds if the restriction operator acts on radial functions.
The main purpose of this paper provides general properties of varieties for which the restriction conjecture holds for all radial test functions.

\section{Statement of main results} While we do not know how to improve Theorem \ref{Lewko}, we are able to show that if the test functions are radial functions, then the $L^p-L^r$ restriction estimates for paraboloids hold for the conjectured exponents given in (\ref{necessary}). In fact, we shall prove more strong results.
 In order to clearly state our main theorems, let us introduce certain definitions and notation.
For each $m=(m_1,\dots, m_d)\in \mathbb F_q^d,$  define 
$$ \|m\|=m_1^2+\dots +m_d^2.$$
We say that a function $f:\mathbb F_q^d \to \mathbb C$ is a radial function if 
$$ f(m)=f(n) \quad \mbox{whenever} ~~\|m\|=\|n\|.$$
For each $j\in \mathbb F_q$, we define
\begin{equation}\label{defsphere} S_j^{d-1}=\{m=(m_1,\dots, m_d)\in \mathbb F_q^d: m_1^2+\dots+m_d^2=j\}\end{equation}
which will be named as the sphere with $j$ radius.
\begin{definition} We write $R_{rad}(p\to r)\lesssim 1$ if the restriction estimate (\ref{restriction}) holds for all radial functions $f:\mathbb F_q^d \to \mathbb C.$
\end{definition}

\subsection{Restriction results on radial functions}   Our first result below shows that  the restriction operators acting on radial functions have quite  good mapping properties.
\begin{theorem} Let d$\sigma$ be the normalized surface measure on an algebraic variety $V\subset \mathbb F_{q*}^d$ with $|V|\sim q^{d-1}.$ Then, we have
\begin{equation}\label{oddresult} R_{rad}\left(\frac{2d}{d+1}\to2\right) \lesssim 1\quad\mbox{for}~~ d\geq 3~~\mbox{odd}\end{equation}
and
\begin{equation}\label{evenresult} R_{rad}\left(\frac{2d-2}{d}\to \frac{2(d-1)^2}{d^2-2d}\right) \lesssim 1\quad \mbox{for}~~ d\geq 4~~\mbox{even}.\end{equation}
\end{theorem}\label{main1}
Using the nesting properties of $L^p$-norms and interpolating (\ref{oddresult}) with the trivial $L^1-L^\infty$, we see that the necessary conditions (\ref{conjecture}) for $A(p\to r)\lesssim 1$ are sufficient  for $R_{rad}(p\to r)\lesssim 1$ if the variety $V$ with $|V|\sim  q^{d-1}$  lies in odd dimensional vector spaces  finite fields.  Notice that the result of Theorem \ref{main1} in even dimensions is weaker than that in odd dimensions. However, the following theorem shows that if the variety $V$  does not contain a lot of elements in the sphere with zero radius, then the result in even dimensions can be improved to that in odd dimensions. 
\begin{theorem}\label{main2}
 Let d$\sigma$ be the normalized surface measure on an algebraic variety $V\subset \mathbb F_{q*}^d$ with $|V|\sim q^{d-1}.$  Suppose that $|V\cap S_0^{d-1}|\lesssim q^{\frac{d^2-d-1}{d}}.$ Then
$$R_{rad}\left(\frac{2d}{d+1}\to2\right) \lesssim 1\quad\mbox{for}~~ d\geq 3.$$
\end{theorem}
It seems that if the algebraic vareity $V$ does not contain $S_0^{d-1}$, the sphere of zero radius, then the conclusion of Theorem \ref{main2} holds.  For example, if $V=\{x\in \mathbb F_{q*}^d: x_1^2+\cdots+x_{d-1}^2=x_d\}$ is the paraboloid or $V=\{x\in \mathbb F_{q*}^d: x_1+\cdots+x_d=0\}$ is the plane, then $|V\cap S_0^{d-1}|\lesssim q^{d-2}< q^{\frac{d^2-d-1}{d}}.$ In this case, we threfore obtain the conclusion of Theorem \ref{main2}. This fact is very interesting in that the Fourier transform of  radial functions can be meaningfully restricted to the plane.

\section{ Fourier decay estimates on spheres}
Since the Fourier transform of a radial function can be written as a linear combination of the Fourier transforms on spheres,  the Fourier decay estimates on spheres shall play a crucial role in proving our results.
In this section, we go over the decay properties of the Fourier transform on spheres $S_j$ in $\mathbb F_q^d.$
We begin with reviewing the classical exponential sum estimates.
For each $a\in  \mathbb F_q^*,$ the Gauss sum is defined by 
$$ G_a:=\sum_{s\in \mathbb F_q^*} \eta{(s)} \chi(as),$$
where $\eta$ denotes the quadratic character of $\mathbb F_q^*.$ The Kloosterman sum is given by
$$ K(a,b):=\sum_{s\in \mathbb F_q^*} \chi(as+ bs^{-1})\quad \mbox{for}~~ a, b\in \mathbb F_q.$$
In addition, recall that  the Sali\'e sum is the exponential sum given by
$$ S(a,b):= \sum_{s\in \mathbb F_q^*} \eta{(s)} \chi(as+bs^{-1}) \quad\mbox{for}~~ a, b\in \mathbb F_q.$$
It is well known that $|G_a|=\sqrt{q} $ for $a\in \mathbb F_q^*,$  $ |K(a,b)|\leq 2 \sqrt{q}$ for $ab\neq 0,$ and $|S(a,b)|\leq 2 \sqrt{q}$ for $a, b\in \mathbb F_q$ (see p.193 in \cite{LN97} and pp.322-323 in \cite{IK04}).  In terms of the aforementioned exponential sums, the authors in \cite{IK10} expressed  the Fourier transforms on the spheres  in $\mathbb F_{q^*}^d$ with the normalized counting measure $dx$.
Modifying the normalizing factor,  one can also do the same thing for the Fourier transform on the spheres in the space $\mathbb F_q^d$ with the counting measure $dm.$
\begin{lemma}\label{lem1} 
Let $S_j^{d-1}$ be the sphere in $\mathbb F_{q}^d$ defined as in (\ref{defsphere}).
Then for any $x\in \mathbb F_{q*}^d,$ we have
$$ \widehat{S_j^{d-1}}(x) =\left\{\begin{array}{ll} q^{d-1} \delta_0(x) + q^{-1} G^d K(-j, -4^{-1}\|x\|)\quad&\mbox{for}~~d\geq 2 \quad\mbox{even}\\
                                                                                        q^{d-1} \delta_0(x) + q^{-1} G^d S(-j, -4^{-1}\|x\|)\quad&\mbox{for}~~d\geq 3 \quad\mbox{odd}\end{array}\right.$$
where $\delta_0(x)=1$ if $x=(0, \ldots, 0)$ and $\delta_0(x)=0$ otherwise.
\end{lemma}
\begin{proof} We follows the same arguments as in Lemma 4 in  \cite{IK10}.
By the definition of the Fourier transform of a function on $\mathbb F_q^d$  with the counting measure $dm$,  if $S^{d-1}_j\subset \mathbb F_q$ and $x\in \mathbb F_{q^*}^d,$ then
$$\widehat{S_j^{d-1}}(x) =\sum_{m\in S^{d-1}_j} \chi(-x\cdot m) = \sum_{m\in \mathbb F_q^d} \chi(-x\cdot m) \delta_0(\|m\|-j).$$ Applying the orthogonality relation of $\chi$, we can write
$$ \delta_0(\|m\|-j)=q^{-1}\sum_{s\in \mathbb F_q} \chi(s(\|m\|-j))\quad \mbox{for}~~x\in \mathbb F_{q^*}^d.$$
It therefore follows that
$$\widehat{S_j^{d-1}}(x)=q^{-1}\sum_{m\in \mathbb F_q^d} \chi(-m\cdot x) + q^{-1} \sum_{s\in \mathbb F_q^*}\chi(-js) \left(\sum_{m\in \mathbb F_q^d} \chi(s\|m\|-x\cdot m) \right)$$
\begin{equation}\label{E1}=q^{d-1}\delta_0(x) + q^{-1}  \sum_{s\in \mathbb F_q^*}\chi(-js) \prod_{k=1}^d \sum_{m_k\in \mathbb F_q} \chi(sm_k^2-x_km_k). \end{equation}
Completing the square and using a change of variables, it follows that
\begin{equation}\label{E2}\sum_{m_k\in \mathbb F_q} \chi(sm_k^2-x_km_k) = \chi( -x_k^2/(4s^2))\sum_{m_k\in \mathbb F_q} \chi(sm_k^2) \quad \mbox{for}~~k=1,2,\dots, d.\end{equation}
Let $A=\{t\in \mathbb F_q^*: t~~\mbox{is a square number}\}$ and  observe that for each $s\in \mathbb F_q^*,$ 
 $$\sum_{t\in \mathbb F_q} \chi(st^2)=1+ \sum_{t\in \mathbb F_q^*}  \chi(st^2) = 1+ \sum_{t\in A} 2 \chi(st)$$
 $$=1+\sum_{t\in \mathbb F_q^*} (1+\eta(t)) \chi(st)= \sum_{t\in \mathbb F_q^*}\eta(t)\chi(st)= \eta(s) G_1.$$
Applying this equality to (\ref{E2}),  it follows from the equality (\ref{E1}) that
$$\widehat{S_j^{d-1}}(x)=q^{d-1}\delta_0(x) + q^{-1}G^d \sum_{s\in \mathbb F_q^*} \eta^d(s)\chi\left(-js+\frac{\|x\|}{-4s}\right).$$
Since $\eta^d=1$ for $d\geq 2$ even, and $\eta^d=\eta$ for $d\geq 3$ odd,  the statement of Lemma \ref{lem1} follows immediately from the definitions of the Kloosterman sum and  the Sali\'e sum.
\end{proof} 

The following corollary can be obtained by applying the estimates of the Gauss sum $G$, the Kloosterman $K(a,b),$ and  the Sali\'e sum $S(a,b)$ to Lemma \ref{lem1}.
\begin{corollary}\label{cor1} Let $d\geq 3$ be an integer. Then,
\begin{equation}\label{size}\widehat{S^{d-1}_j}(0,\dots,0)=|S^{d-1}_j|\sim q^{d-1}\quad\mbox{for}~~j\in \mathbb F_q.\end{equation}
If $d\geq 2$ and $x\in \mathbb F_{q*}^d\setminus \{(0,\dots,0)\}$  then
\begin{equation}\label{decay}|\widehat{S^{d-1}_j}(x)|\lesssim \left\{\begin{array}{ll} q^{\frac{d-1}{2}}&\quad\mbox{for} ~~d~ \mbox{odd},~  j\in \mathbb F_q\\
                                                                                               q^{\frac{d-1}{2}}&\quad\mbox{for} ~~d~ \mbox{even},~  j\neq 0\\ 
                                                                                                    q^{\frac{d}{2}}&\quad\mbox{for} ~~d~ \mbox{even}, ~ j=0. \end{array}\right.\end{equation} 
In particular, if $\|x\|\neq 0$ and $d\geq 4$ is even, then 
\begin{equation}\label{gooddecay} 
|\widehat{S_0^{d-1}}(x)|= q^{\frac{d-2}{2}}.
\end{equation}                   
\end{corollary}
\begin{proof}  First, let us prove (\ref{size}).  It is clear from the definition of the Fourier transform that $ \widehat{S^{d-1}_j}(0,\dots,0)=\sum_{m\in \mathbb F_q^d} \chi(m\cdot (0,\dots,0)) S_j^{d-1}(m)=|S^{d-1}_j|.$
On the other hand,  it follows from Lemma \ref{lem1} that 
$$\widehat{S_j^{d-1}}(0,\dots,0) =\left\{\begin{array}{ll} q^{d-1} + q^{-1} G^d K(-j, 0)\quad&\mbox{for}~~d\geq 2 \quad\mbox{even}\\
                                                                                        q^{d-1} + q^{-1} G^d S(-j, 0)\quad&\mbox{for}~~d\geq 3 \quad\mbox{odd}\end{array}\right.$$

Since $|G|=\sqrt{q},~  |K(-j, 0)|\leq q$ for $j\in \mathbb F_q,$ and $|S(-j, 0)|\leq 2\sqrt{q}$ for $j\in \mathbb F_q,$ we see that if $d\geq 3$, then $q^{d-1} + q^{-1} G^d K(-j, 0)\sim  q^{d-1} + q^{-1} G^d S(-j, 0)\sim q^{d-1}.$
Thus (\ref{size}) holds.  Next, using Lemma \ref{lem1},  the conclusion (\ref{decay}) is an immdeate consequence from facts that
$ |G|\sqrt{q}, ~|K(a,b)|\leq 2\sqrt{q}$ if $ab\neq 0, ~ |K(a,b)|\leq q $ if $ab=0$,  and $|S(a,b)|\leq 2\sqrt{q}$ if $a,b\in \mathbb F_q.$  Finally,  the equality (\ref{gooddecay}) follows from Lemma \ref{lem1} and the observations that
$|G|=\sqrt{q}~,  |K(0,b)|=1 $ for $b\neq 0.$ 
\end{proof}
\section{Proofs of Theorem \ref{main1} and Theorem \ref{main2}}
In this section we shall complete the proofs of main results on restriction operators acting on radial functions.
First we will derive sufficient conditions for $R_{rad}(p\to r)\lesssim 1.$
We aim to find certain conditions on $1\leq p, r\leq \infty$ such that
$$ \|\widehat{f}\|_{L^r(V, d\sigma)}\lesssim \|f\|_{L^p(\mathbb F_q^d, dm)} \quad\mbox{for all radial functions}~~f:\mathbb F_q^d \to \mathbb C.$$
Without loss of generality, we may assume that $f$ is a nonnegative, radial function on $\mathbb F_q^d.$ Therefore, we can write
$$ f(m)=M_j \geq 0 \quad \mbox{if}~~m\in S_j^{d-1} ~~\mbox{for some}~j\in \mathbb F_q.$$
By multiplying a normalizing constant, we may also assume that $$ \|f\|_{L^{p}(\mathbb F_q^d, dm)} =1.$$
It therefore follows that
 $$1=\|f\|^{p}_{L^{p}(\mathbb F_q^d, dm)}=\sum_{m\in \mathbb F_q^d} |f(m)|^{p}=\sum_{j\in \mathbb F_q}\sum_{m\in S^{d-1}_j} M_j^{p}
=\sum_{j\in \mathbb F_q}M_j^p |S^{d-1}_j|$$
Since $|S^{d-1}_j|\sim q^{d-1}$ for $j\in \mathbb F_q$, we have
\begin{equation}\label{ass1}\sum_{j\in \mathbb F_q}M_j^{p}\sim q^{1-d}.\end{equation}
With assumptions above on the radial function $f$, it suffices to find certain conditions on $1\leq p,r\leq \infty$ such that
\begin{equation}\label{con1}\|\widehat{f}\|_{L^{r }(V, d\sigma)}^r :=\frac{1}{|V|}\sum_{x\in V} |\widehat{f}(x)|^r \lesssim 1.\end{equation}
Since $\widehat{f}(x)=\sum\limits_{m\in \mathbb F_q^d} \chi(-m\cdot x) f(m)
= \sum\limits_{j\in \mathbb F_q} \sum\limits_{m\in S_j^{d-1}} \chi(-m\cdot x) M_j,$ It follows that
$$\|\widehat{f}\|_{L^{r }(V, d\sigma)}^r=\frac{1}{|V|} \sum_{x\in V} \left| \sum_{j\in \mathbb F_q} M_j \widehat{S_j}(x)\right|^r$$
$$=\frac{1}{|V|} \sum_{x\in V\setminus\{(0,\dots,0)\}} \left| \sum_{j\in \mathbb F_q} M_j \widehat{S_j^{d-1}}(x)\right|^r + \frac{1}{|V|} \left| \sum_{j\in \mathbb F_q} M_j \widehat{S^{d-1}_j}(0,\dots,0)\right|^r .$$
Since $M_j\geq 0, |V|\sim q^{d-1},$ and $\widehat{S_j^{d-1}}(0,\dots,0) =|S_j^{d-1}|\sim q^{d-1}$ for $d\geq 3,$ we see that
$$\|\widehat{f}\|_{L^{r }(V, d\sigma)}^r \sim\frac{1}{q^{d-1}} \sum_{x\in V\setminus\{(0,\dots,0)\}} \left| \sum_{j\in \mathbb F_q} M_j \widehat{S_j^{d-1}}(x)\right|^r
+   \frac{q^{r(d-1)}}{q^{d-1}} \left( \sum_{j\in \mathbb F_q} M_j\right) ^r. $$
From Holder's inequality and (\ref{ass1}), observe that 
\begin{equation}\label{Mtest} \left( \sum_{j\in \mathbb F_q} M_j\right) ^r \leq \left(\sum_{j\in \mathbb F_q} 1^{p^\prime}\right)^{\frac{r}{p^\prime}}~\left( \sum_{j\in \mathbb F_q} M_j^{p}\right)^{\frac{r}{p}}
\sim q^{r(1-\frac{d}{p})},\end{equation}
where $p^\prime$ denotes the Holder conjugate of $p$ , that is $ p^\prime= p/(p-1).$
It follows that 
$$\|\widehat{f}\|_{L^{r }(V, d\sigma)}^r \lesssim \frac{1}{q^{d-1}} \sum_{x\in V\setminus\{(0,\dots,0)\}} \left| \sum_{j\in \mathbb F_q} M_j \widehat{S_j^{d-1}}(x)\right|^r+ 
q^{rd(1-\frac{1}{p})-d+1}.$$
Combining this with (\ref{con1}),  the sufficiet conditions on $1\leq p,r\leq \infty$ for $R_{rad}(p\to r)\lesssim 1$ are given by 
\begin{equation}\label{suf1}\frac{1}{q^{d-1}} \sum_{x\in V\setminus\{(0,\dots,0)\}} \left| \sum_{j\in \mathbb F_q} M_j \widehat{S_j^{d-1}}(x)\right|^r \lesssim 1
\end{equation}
and
\begin{equation}\label{suf2} rd(1-\frac{1}{p})-d+1\leq 0.
\end{equation}

\subsection{Proof of the first part of conclusions in Theorem \ref{main1}}
We prove  (\ref{oddresult}) of Theorem \ref{main1}. Namely, we shall prove that if $d\geq 3$ is odd, then
$$  \|\widehat{f}\|_{L^{2 }(V, d\sigma)} \lesssim \|f\|_{L^{  \frac{2d}{d+1}}(\mathbb F_q^d, dm)} \quad\mbox{for all radial functions}~~f:\mathbb F_q^d \to \mathbb C.$$
With $p= \frac{2d}{d+1}$ and $r=2$, it is enough to show that  the inequalities (\ref{suf1}), (\ref{suf2}) hold.
The inequaltiy (\ref{suf2}) follows immediately from a simple calculation that if $p= \frac{2d}{d+1}$ and $r=2,$ then $ rd(1-\frac{1}{p})-d+1=0.$ 
To prove the inequality (\ref{suf1}), recall from (\ref{decay}) in Corollary \ref{cor1} that if $d\geq 3$ is odd, then
$$ |\widehat{S_j^{d-1}}(x)|\lesssim q^{\frac{d-1}{2}} \quad \mbox{for}~~ j\in \mathbb F_q,  ~x\neq (0,\dots, 0).$$
From this fact and (\ref{Mtest}), the inequality (\ref{suf1}) can be established for $p= \frac{2d}{d+1}$ and $r=2$ by the following observation:
\begin{equation}\label{arg} \frac{1}{q^{d-1}} \sum_{x\in V\setminus\{(0,\dots,0)\}} \left| \sum_{j\in \mathbb F_q} M_j \widehat{S_j^{d-1}}(x)\right|^2\lesssim \sum_{x\in V\setminus\{(0,\dots,0)\}} \left( \sum_{j\in \mathbb F_q} M_j \right)^2 \end{equation}
$$ \lesssim |V|\left( \sum_{j\in \mathbb F_q} M_j \right)^2 \lesssim q^{d-1} q^{2(1-\frac{d+1}{2})}= 1, $$
where  we also used that $ M_j\geq 0$ and $|V|\sim q^{d-1}.$ 

\subsection{Proof of the second part of conclusions in Theorem \ref{main1}}
We prove  (\ref{evenresult}) of Theorem \ref{main1}. When $d\geq 4$ is even, we must prove $R_{rad}(p\to r)\lesssim 1$ for $p=\frac{2d-2}{d}$ and $r=\frac{2(d-1)^2}{d^2-2d}.$
As mentioned before,  it suffices to prove the inequalities (\ref{suf1}), (\ref{suf2}) for $p=\frac{2d-2}{d}$ and $r=\frac{2(d-1)^2}{d^2-2d}.$
In this case, the inequality (\ref{suf2}) is clearly true because  $rd(1-\frac{1}{p})-d+1=0.$
To prove the inequality (\ref{suf1}),  recall from (\ref{decay}) in Corollary \ref{cor1} that if $d\geq 4$ is even and $x\neq (0,\dots,0)$, then
$$|\widehat{S^{d-1}_j}(x)|\lesssim \left\{\begin{array}{ll}       q^{\frac{d-1}{2}}&\quad\mbox{for} ~~  j\neq 0\\ 
                                                                                                    q^{\frac{d}{2}}&\quad\mbox{for} ~~ j=0. \end{array}\right.$$    
From this fact, the left part of the inequality (\ref{suf1}) can be estimated as follows.
$$\frac{1}{q^{d-1}} \sum_{x\in V\setminus\{(0,\dots,0)\}} \left| \sum_{j\in \mathbb F_q} M_j \widehat{S_j^{d-1}}(x)\right|^r $$
$$\lesssim \frac{1}{q^{d-1}} \sum_{x\in V\setminus\{(0,\dots,0)\}} \left|  M_0 \widehat{S_0^{d-1}}(x)\right|^r + \frac{1}{q^{d-1}} \sum_{x\in V\setminus\{(0,\dots,0)\}}\left| \sum_{j\neq 0} M_j \widehat{S_j^{d-1}}(x)\right|^r$$
$$\lesssim  \frac{1}{q^{d-1}} q^{\frac{rd}{2}} M_0^r \left(\sum_{x\in V\setminus\{(0,\dots,0)\}} 1\right) 
+ \frac{1}{q^{d-1}} q^{\frac{r(d-1)}{2}}\left(\sum_{j\neq 0} M_j\right)^r 
\left( \sum_{x\in V\setminus\{(0,\dots,0)\}} 1\right)$$
$$\leq q^{\frac{rd}{2}} M_0^r +q^{\frac{r(d-1)}{2}}\left(\sum_{j\in \mathbb F_q} M_j\right)^r\lesssim q^{\frac{rd}{2}} M_0^r + q^{\frac{r(d-1)}{2}} q^{r(1-\frac{d}{p})}$$
where the last inequality follows from (\ref{Mtest}).
Since $M_j\geq 0$ for $j\in \mathbb F_q,$ it is clear from (\ref{ass1}) that  $ M_0\lesssim q^{\frac{1-d}{p}}.$ Thus, we have
\begin{equation}\label{mo}
M_0^r \lesssim  q^{\frac{r(1-d)}{p}}.\end{equation}
 We threrefore see that if $p=\frac{2d-2}{d}$ and $r=\frac{2(d-1)^2}{d^2-2d},$ then
$$\frac{1}{q^{d-1}} \sum_{x\in V\setminus\{(0,\dots,0)\}} \left| \sum_{j\in \mathbb F_q} M_j \widehat{S_j^{d-1}}(x)\right|^r \lesssim 
q^{\frac{rd}{2}+ \frac{r(1-d)}{p}} + q^{\frac{r(d-1)}{2}+r(1-\frac{d}{p})}=q^0+ q^{-\frac{d-1}{d(d-2)}} \lesssim 1,$$  which proves the inequality (\ref{suf1}).
Hence, we complete the proof.

\subsection{Proof of  Theorem \ref{main2}}
 Let d$\sigma$ be the normalized surface measure on an algebraic variety $V\subset \mathbb F_{q*}^d$ with $|V|\sim q^{d-1}.$ 
Assuming that $|V\cap S_0^{d-1}|\lesssim q^{\frac{d^2-d-1}{d}},$ we aim to prove that
$$R_{rad}\left(\frac{2d}{d+1}\to2\right) \lesssim 1\quad\mbox{for}~~ d\geq 3.$$
In the case when $d\geq 3$ is odd, this statement was already proved in the first part of Theorem \ref{main1} with much weaker assumptions.
Thus, we may assume that $d\geq 4$ is even.
Suppose that
\begin{equation}\label{intersection}  |V\cap S_0^{d-1}|\lesssim q^{\frac{d^2-d-1}{d}}.\end{equation}
As before, our task is to prove that the inequalities (\ref{suf1}), (\ref{suf2}) hold for $p=\frac{2d}{d+1}$ and $r=2.$
As mentioned before, the inequality (\ref{suf2}) clearly holds for  $p=\frac{2d}{d+1}$ and $r=2.$
To prove the inequality (\ref{suf1}), let
$$\mbox{L}:=\frac{1}{q^{d-1}} \sum_{x\in V\setminus\{(0,\dots,0)\}} \left| \sum_{j\in \mathbb F_q} M_j \widehat{S_j^{d-1}}(x)\right|^r $$
and show that $L\lesssim 1.$
It follows that
$$\mbox{L}\lesssim \frac{1}{q^{d-1}} \sum_{x\in V\setminus\{(0,\dots,0)\}} M_0^r  \left|  \widehat{S_0^{d-1}}(x)\right|^r + 
 \frac{1}{q^{d-1}} \sum_{x\in V\setminus\{(0,\dots,0)\}}\left| \sum_{j\neq 0} M_j \widehat{S_j^{d-1}}(x)\right|^r:=\mbox{R} + \mbox{M}. $$
It suffices to prove that  for $p= \frac{2d}{d+1}$ and $r=2$,
\begin{equation}\label{goal1}\mbox{R}=\frac{1}{q^{d-1}} \sum_{x\in V\setminus\{(0,\dots,0)\}} M_0^r  \left|  \widehat{S_0^{d-1}}(x)\right|^r \lesssim 1
\end{equation}
and
\begin{equation}\label{goal2}\mbox{M}= \frac{1}{q^{d-1}} \sum_{x\in V\setminus\{(0,\dots,0)\}}\left| \sum_{j\neq 0} M_j \widehat{S_j^{d-1}}(x)\right|^r\lesssim 1.
\end{equation}
The inequality (\ref{goal2}) follows immediately from the same argument in (\ref{arg}).
To prove the inequality (\ref{goal1}), we write
$$ \mbox{R}= \frac{1}{q^{d-1}} \sum_{x\in V\setminus\{(0,\dots,0)\}: \|x\|=0} M_0^r  \left|  \widehat{S_0^{d-1}}(x)\right|^r  
+\frac{1}{q^{d-1}} \sum_{x\in V\setminus\{(0,\dots,0)\}:\|x\|\neq 0} M_0^r  \left|  \widehat{S_0^{d-1}}(x)\right|^r.$$
Since $d\geq 4$ is even,  the application of  (\ref{decay}) and (\ref{gooddecay}) in Corollary \ref{cor1} yield that
$$ \mbox{R}\lesssim \frac{M_0^r }{q^{d-1}}q^{\frac{rd}{2}} |V\cap S^{d-1}_0| +  \frac{M_0^r }{q^{d-1}}q^{\frac{r(d-2)}{2}} |V|.$$
By (\ref{mo}) and our assumption that $|V\cap S^{d-1}_0|\lesssim q^{\frac{d^2-d-1}{d}}$, we see that if $p=\frac{2d}{d+1}$ and $r=2$, then
$$ \mbox{R}\lesssim q^{\frac{r(1-d)}{p}+ \frac{rd}{2}-\frac{1}{d}} + q^{\frac{r(1-d)}{p}+\frac{r(d-2)}{2}}=q^0 + q^{\frac{1-2d}{d}} \lesssim 1 .$$
 The proof of Theorem \ref{main2} is complete.

\end{document}